\documentclass[12pt]{amsart}  

\usepackage[latin1]{inputenc}
\usepackage{amsmath} 
\usepackage{amsfonts}
\usepackage{amssymb}
\usepackage{stmaryrd}
\usepackage{latexsym} 
\usepackage{graphicx}
\usepackage{subfigure}
\usepackage{color}
\usepackage{hyperref}
\usepackage{verbatim}
\usepackage[all]{xy}
\usepackage{graphics}
\usepackage{pdfsync}
\usepackage{ulem}

\theoremstyle{definition}
\newtheorem{theorem}{Theorem}[section]

\newtheorem{corollary}[theorem]{Corollary}

\newtheorem{example}[theorem]{Example}

\theoremstyle{remark}
\newtheorem{remark}[theorem]{Remark}

\numberwithin{equation}{section}
\newcommand{\bC}{\mathbb{C}}

\newcommand{\bR}{\mathbb{R}}

\newcommand{\prob}{\mathbb{P}}

\newcommand{\sgrp}{\mathfrak{S}}

\newcommand{\Sym}{\ensuremath{\operatorname{Sym}}}

\newcommand{\QSym}{\ensuremath{\operatorname{QSym}}}
\newcommand{\qs}{\ensuremath{\mathcal{S}}}

\newcommand{\llex}{<_{\mathit{lex}}}
\newcommand{\glex}{>_{\mathit{lex}}}
\newcommand{\lrlex}{<_{\mathit{revlex}}}

\newcommand{\mmax}{\mathit{max}}
\newcommand{\btr}{\blacktriangleleft}
\newcommand{\set}{\mathrm{set}} 
\newcommand{\des}{\mathrm{Des}} 
\newcommand{\comp}{\mathrm{comp}} 
\newcommand{\cont}{\mathrm{content}} 
\newcommand{\refines}{\prec}   
\newcommand{\suchthat}{\;|\;}
\newcommand{\spam}{\operatorname{span}}
\newcommand{\Bcone}{B^+_{\text{cone}}}
\newcommand{\Acone}{A^+_{\text{cone}}}
\newcommand{\Bslice}{B^+_{\text{slice}}}
\newcommand{\Aslice}{A^+_{\text{slice}}}

\usepackage[colorinlistoftodos]{todonotes}

\newlength\cellsize 
\savebox2{%
\begin{picture}(15,15)
\put(0,0){\line(1,0){15}}
\put(0,0){\line(0,1){15}}
\put(15,0){\line(0,1){15}}
\put(0,15){\line(1,0){15}}
\end{picture}}
\newcommand\cellify[1]{\def\thearg{#1}\def\nothing{}%
\ifx\thearg\nothing
\vrule width0pt height\cellsize depth0pt\else
\hbox to 0pt{\usebox2\hss}\fi%
\vbox to 15\unitlength{
\vss
\hbox to 15\unitlength{\hss$#1$\hss}
\vss}}
\newcommand\tableau[1]{\vtop{\let\\=\cr
\setlength\baselineskip{-16000pt}
\setlength\lineskiplimit{16000pt}
\setlength\lineskip{0pt}
\halign{&\cellify{##}\cr#1\crcr}}}
\savebox3{%
\begin{picture}(15,15)
\put(0,0){\line(1,0){15}}
\put(0,0){\line(0,1){15}}
\put(15,0){\line(0,1){15}}
\put(0,15){\line(1,0){15}}
\end{picture}}
\newcommand\expath[1]{%
\hbox to 0pt{\usebox3\hss}%
\vbox to 15\unitlength{
\vss
\hbox to 15\unitlength{\hss$#1$\hss}
\vss}}
\newcommand\bas[1]{\omit \vbox to \cellsize{ \vss \hbox to \cellsize{\hss$#1$\hss} \vss}}

\begin{document}

\title[The probability of positivity]{The probability of positivity in symmetric and quasisymmetric functions}

\author{Rebecca Patrias}
\address{
Laboratoire de Combinatoire et d'Informatique Math\'{e}matique, 
Universit\'{e} du Qu\'{e}bec \`{a} Montr\'{e}al,
Montr\'{e}al QC H3C 3P8, Canada}
\email{patriasr@lacim.ca}

\author{Stephanie van Willigenburg}
\address{
 Department of Mathematics,
 University of British Columbia,
 Vancouver BC V6T 1Z2, Canada}
\email{steph@math.ubc.ca}


\begin{abstract} Given an element in a finite-dimensional real vector space, $V$, that is a nonnegative linear combination of basis vectors for some basis $B$, we compute the probability that it is furthermore a nonnegative linear combination of basis vectors for a second basis, $A$. We then apply this general result to combinatorially compute the probability that a symmetric function is Schur-positive ({recovering the recent result of Bergeron--Patrias--Reiner}), $e$-positive or $h$-positive. Similarly we compute the probability that a quasisymmetric function is quasisymmetric Schur-positive or fundamental-positive. In every case we conclude that the probability tends to zero as the degree of a function tends to infinity. 
\end{abstract}

\maketitle
\tableofcontents
\section{Introduction}\label{sec:intro} The subject of when a symmetric function is Schur-positive, that is, a nonnegative linear combination of Schur functions, is an active area of research. If a homogeneous symmetric function of degree $n$ is Schur-positive, then it is the image of some representation of the symmetric group $\sgrp _n$ under the Frobenius characteristic map. Furthermore, if it is a polynomial, then it is the character of a polynomial representation of the general linear group $GL(n,\bC)$. Consequently, much effort has been devoted to determining when the difference of two symmetric  functions is Schur-positive, for example \cite{BBR06, FFLP05, KWvW08,  Kir04, lpp,  McN08, McvW09b, Oko97}. While this question is still wide open in full generality, there exist well-known examples of Schur-positive functions.  These include the product of two Schur functions, skew Schur functions, and the chromatic symmetric function of the incomparability graph of $(3+1)$-free  posets \cite{Gasharov}, the latter of which are further conjectured to be $e$-positive, that is, a nonnegative linear combination of elementary symmetric functions \cite{StanleyStembridge}. One other well-known example that is conjectured to be Schur-positive is the bigraded Frobenius characteristic of the space of diagonal harmonics \cite{HHLRU}, which was recently proved to be fundamental-positive, namely, a nonnegative linear combination of fundamental quasisymmetric functions \cite{CarlssonMellit}. This result is better known as the proof of the shuffle conjecture. 

Quasisymmetric functions, a natural generalization of symmetric functions,   are further related to  positivity via representation theory since the 1-dimensional representations of the 0-Hecke algebra map to fundamental quasisymmetric functions under the quasisymmetric characteristic map \cite{DKLT}. There also exist 0-Hecke modules whose quasisymmetric characteristic map images are quasisymmetric Schur functions \cite{0Hecke}. Additionally, if a quasisymmetric function is both symmetric and a nonnegative linear combination of quasisymmetric Schur functions, then it is Schur-positive \cite{SSQSS}. While nonnegative linear combinations of quasisymmetric functions are not as extensively studied, some progress has been made in this direction, for example \cite{AlexSulz, BLvW, SSQSS, LamP, McN14}, and this area is ripe for study. 

This paper is structured as follows. {In Theorem~\ref{thm:genthm}, we calculate the probability that an element of a vector space that is a nonnegative linear combination of basis elements is also a nonnegative linear combination of the elements of a second basis, where the bases satisfy certain conditions.} In Section~\ref{sec:probsym}, we then apply this theorem and compute the probability that a symmetric function is Schur-positive or $e$-positive in Corollaries~\ref{cor:smprob}, \ref{cor:ehsprob}, \ref{cor:emprob}. We show that these probabilities tend to 0 as the degree of the function tends to infinity in Corollaries~\ref{cor:sm0},  \ref{cor:ehs0}, \ref{cor:em0}. We then apply Thoerem~\ref{thm:genthm} again in Section~\ref{sec:probqsym} to compute the probability that a quasisymmetric function is quasisymmetric Schur-positive or fundamental-positive in Corollaries~\ref{cor:SMprob}, \ref{cor:SFprob}, \ref{cor:FMprob}, and similarly show these probabilities tend to 0 in Corollaries~\ref{cor:SM0}, \ref{cor:SF0}, \ref{cor:FM0}.

\section{The probability of vector positivity}\label{sec:probvec}
Let $V$ be a finite-dimensional real vector space with bases $A=\{A_0,\ldots,A_d\}$ and $B=\{B_0,\ldots,B_d\}$, and suppose further that
\[A_j=\sum_{i\leq j} a_i^{(j)}B_i,\] where $a_j^{(j)}=1$ and $a_i^{(j)}\geq 0$. In particular, note that $A_0=B_0$. We say that $f\in V$ is \emph{$A$-positive} (respectively, \emph{$B$-positive}) if $f$ is a nonnegative linear combination of $\{A_0,\ldots,A_d\}$ (respectively, $\{B_0,\ldots,B_d\}$). We would like to answer the following question: What is the probability that if $f\in V$ is $B$-positive, then it is furthermore $A$-positive? We denote this probability by $\prob(A_i \suchthat B_i)$ and note that any $A$-positive $f\in V$ will also necessarily be $B$-positive. 

In order to calculate $\prob (A_i \suchthat B_i)$, observe that any $B$-positive $f\in V$ can be written as 
\[f=\sum_{i=0}^d b_iB_i,\]
where each $b_i\geq 0$, and the set of all $B$-positive elements of $V$ forms a cone
\[\Bcone=\left\{\sum_{i=0}^d b_iB_i \suchthat b_i\in\mathbb{R}_{\geq0}\right\}.\] Inside the cone $\Bcone$ is the cone of $A$-positive elements of $V$
\[\Acone=\left\{\sum_{i=0}^db_iB_i \suchthat b_i\in\mathbb{R}_{\geq0}\text{ and the expression is $A$-positive}\right\}.\] 

We define $\prob(A_i \suchthat B_i)$ to be the ratio of the volume of the slice of $A^+_{\text{cone}}$ defined by \[\Aslice=\left\{\sum_{i=0}^d b_iB_i \suchthat b_i\in\mathbb{R}_{\geq0}\text{, the expression is $A$-positive, and }\sum_{i=0}^db_i=1\right\}\] to the volume of the slice of $B^+_{\text{cone}}$ \[\Bslice=\left\{\sum_{i=0}^d b_iB_i \suchthat b_i\in\mathbb{R}_{\geq0}\text{ and }\sum_{i=0}^db_i=1\right\}.\] We could, equivalently, replace ``1'' in both definitions with any positive real number and obtain the same ratio. Note that this probability will depend on the choice of bases $\{A_0,\ldots,A_d\}$ and $\{B_0,\ldots,B_d\}$; however, each application of the following theorem (see Corollaries~\ref{cor:smprob}, \ref{cor:ehsprob}, \ref{cor:emprob}, \ref{cor:SMprob}, \ref{cor:SFprob}, and \ref{cor:FMprob}) comes with a natural choice of bases, and the asymptotics we explore (see Corollaries~\ref{cor:sm0}, \ref{cor:ehs0}, \ref{cor:em0}, \ref{cor:SM0}, \ref{cor:SF0}, and \ref{cor:FM0})  do not depend on this choice. 

The existence of the general statement below was suggested by F. Bergeron and its proof inspired by a conversation with V. Reiner about Schur-positivity.

\begin{theorem}\label{thm:genthm}
Let $\{A_0,\ldots,A_d\}$ and  $\{B_0,\ldots,B_d\}$ be bases of a finite-dimensional real vector space $V$ such that 
\[A_j=\sum_{i\leq j} a_i^{(j)}B_i,\] where $a_j^{(j)}=1$ and $a_i^{(j)}\geq 0$, so in particular, $A_0=B_0$. Then 
\[\prob(A_i\suchthat B_i)=\prod_{j=0}^d\left(\sum_{i=0}^j a_i^{(j)}\right)^{-1}.\]
\end{theorem}

\begin{proof}
Consider
\[\Bslice=\left\{\sum_{i=0}^d b_iB_i \suchthat b_i\in\mathbb{R}_{\geq0}\text{ and }\sum_{i=0}^db_i=1\right\}\]
and the corresponding slice of $\Acone$

\[\Aslice=\left\{\sum_{i=0}^d b_iB_i \suchthat b_i\in\mathbb{R}_{\geq0}\text{, the expression is $A$-positive, and }\sum_{i=0}^db_i=1\right\}.\]

Note that $\Bslice$ is the simplex determined by vertices $B_0,\ldots,B_d$. Define vectors $v_1,\ldots,v_d$ by $v_i=B_i-B_0$ for $1\leq i\leq d$. Then the volume of $\Bslice$ is by definition \[\frac{1}{d!}\lvert \det(v_1,\ldots,v_d)\rvert. \]

The simplex $\Aslice$ is determined by vertices $\left\{\left(\sum_i a_i^{(j)}\right)^{-1}A_j\right\}_{0\leq j\leq d}$. To find its volume, we first define vectors $w_1,\ldots,w_d$ by
\[w_j=\left(\sum_i a_i^{(j)}\right)^{-1}A_j-A_0.\]
We then see that 
\begin{align*}
w_j&=\frac{1}{\sum_ia_i^{(j)}}A_j-A_0\\
&= \frac{1}{\sum_ia_i^{(j)}}\left(B_j+a^{(j)}_{j-1}B_{j-1}+\cdots+a^{(j)}_0B_0\right)-B_0\\
&=\frac{1}{\sum_ia_i^{(j)}}\left(B_j+a^{(j)}_{j-1}B_{j-1}+\cdots+a^{(j)}_0B_0\right)-\frac{1}{\sum_ia_i^{(j)}}\left(B_0+a^{(j)}_{j-1}B_0+\cdots+a_0^{(j)}B_0\right)\\
&= \frac{1}{\sum_ia_i^{(j)}}\left(v_j+a^{(j)}_{j-1}v_{j-1}+\cdots+a^{(j)}_1v_1\right).
\end{align*}
Thus the volume of $\Aslice$, namely the simplex determined by vertices $\left\{\frac{1}{\sum_ia_i^{(j)}}A_j\right\}$, is 

\begin{align*}\frac{1}{d!}\lvert\det(w_1,\ldots,w_d)\rvert&=\frac{1}{d!}\left\lvert\det\left(\frac{1}{\sum_i a_i^{(1)}}v_1,\frac{1}{\sum_ia^{(2)}_i}(v_2+a^{(2)}_1v_1),\ldots,\frac{1}{\sum_ia_i^{(d)}}(v_d+\cdots+a^{(d)}_1v_1)\right)\right\rvert\\
&=\frac{1}{d!}\prod_{j}\frac{1}{\sum_ia_i^{(j)}} \left\lvert\det(v_1,\ldots,v_d)\right\rvert.
\end{align*} The result now follows, since by definition we have that

$$\prob(A_i \suchthat B_i) =\frac{\mbox{volume of } \Aslice}{\mbox{volume of } \Bslice}.$$ \end{proof}


\section{Probabilities of symmetric function positivity}\label{sec:probsym} Before we define the various symmetric functions that will be of interest to us, we need to recall some combinatorial concepts. A \emph{partition}  $\lambda = (\lambda _1, \ldots , \lambda _k)$  of $n$, denoted by $\lambda \vdash n$, is a list of positive integers whose \emph{parts} $\lambda _i$ satisfy $\lambda _1 \geq \cdots \geq \lambda _k$ and $\sum _{i=1}^k \lambda _i = n$. If there exists $\lambda _{m+1} = \cdots = \lambda _{m+j} = i$, then we often abbreviate this to $i^j$. There exist two total orders on partitions of $n$, which will be useful to us. The first of these is \emph{lexicographic order}, which states that given partitions  $\lambda = (\lambda _1, \ldots , \lambda _k)$ and $\mu = (\mu _1, \ldots , \mu _\ell)$ we say that $\mu$ is lexicographically smaller that $\lambda$, denoted by $\mu \llex \lambda$, if $\mu \neq \lambda$ and the first $i$ for which $\mu _i \neq \lambda _i$ satisfies $\mu _i < \lambda _i$. The second is the closely related \emph{reverse lexicographic order}, where we say that $\mu$ is reverse lexicographically smaller that $\lambda$, denoted by $\mu \lrlex \lambda$ if and only if $\mu \glex \lambda$.

\begin{example}\label{ex:lex} The partitions of 4 in lexicographic order are
$$(1^4)\llex (2,1^2)\llex (2^2)\llex (3,1) \llex (4).$$
\end{example}

Given a partition $\lambda = (\lambda _1, \ldots , \lambda _k)$ and commuting variables $\{x_1, x_2, \ldots \}$, we define the \emph{monomial symmetric function} $m_\lambda$ to be
$$m_\lambda = \sum x_{i_1} ^{\lambda _1} \cdots x_{i_k} ^{\lambda _k}$$where the sum is over all $k$-tuples $(i_1, \ldots , i_k)$ of distinct indices that yield distinct monomials. 

\begin{example}\label{ex:m} We see that $m_{(2,1)} = x_1^2 x_2 + x_2^2 x_1 + x_1^2 x_3 + x_3^2 x_1 + \cdots .$
\end{example}

The set of all monomial symmetric functions forms a basis for the graded algebra of symmetric functions
$$\Sym = \bigoplus _{n\geq 0} \Sym ^n \subseteq \bR [[x_1, x_2, \ldots ]]$$where $\Sym ^0 = \spam \{1\}$ and $\Sym ^n  = \spam \{m_\lambda \suchthat \lambda \vdash n\}$ for $n\geq 1$. Hence each graded piece $\Sym ^n$ for $n\geq 1$ is a finite-dimensional real vector space with basis $ \{m_\lambda \suchthat \lambda \vdash n\}$.  

For our second required basis we need Young diagrams and Young tableaux. Given a partition $\lambda = (\lambda _1, \ldots , \lambda _k)\vdash n$, we call the array of $n$ left-justified boxes with $\lambda _i$ boxes in row $i$ from the top, for $1\leq i \leq k$, the \emph{Young diagram} of $\lambda$, also denoted by $\lambda$. Given a Young diagram we say that $T$ is a \emph{semistandard Young tableau (SSYT)} of \emph{shape} $\lambda$ if the boxes of $\lambda$ are filled with positive integers such that
\begin{enumerate}
\item the entries in each row weakly increase when read from left to right,
\item the entries in each column strictly increase when read from top to bottom.
\end{enumerate}
Two SSYTs of shape $(2,1)$ can be seen below in Example~\ref{ex:sasm}.
Given an SSYT $T$ we define the \emph{content} of $T$, denoted by $\cont(T)$, to be the list of nonnegative integers
$$\cont(T) = (c_1, \ldots , c_{\mmax})$$where $c_i$ is the number of times that $i$ appears in $T$ and $\mmax$ is the largest integer appearing in $T$. We say that an SSYT is of \emph{partition content} if $c_1\geq \cdots \geq c_{\mmax}>0$. With this in mind, if $\lambda$ and $\mu$ are partitions, then we define the \emph{Schur function} $s_\lambda$ to be
\begin{equation}\label{eq:sasm}s_\lambda = m_\lambda + \sum _{\mu\llex \lambda} K_{\lambda\mu} m_\mu\end{equation}where $K_{\lambda\mu}$ is the number of SSYTs, $T$, of shape $\lambda$ and  $\cont(T)=\mu$.

\begin{example}\label{ex:sasm} We see $s_{(2,1)} = m_{(2,1)} + 2 m_{(1,1,1)}$ from the following two SSYTs arising from the nonleading term.
$$\tableau{1&2\\3}\quad \tableau{1&3\\2}$$
\end{example}

We now define the \emph{$i$-th complete homogeneous symmetric function} to be
$$h_i=s_{(i)}$$and if $\lambda = (\lambda _1, \ldots , \lambda _k)$ is a partition then we define the \emph{complete homogeneous symmetric function} $h_\lambda$ to be
$$h_\lambda = h_{\lambda _1}\cdots h_{\lambda _k} = s_{(\lambda _1)}\cdots s_{(\lambda _k)}.$$Similarly, we define the \emph{$i$-th elementary symmetric function} to be
$$e_i=s_{(1^i)}$$and if $\lambda = (\lambda _1, \ldots , \lambda _k)$ is a partition then we define the \emph{elementary symmetric function} $e_\lambda$ to be
$$e_\lambda = e_{\lambda _1}\cdots e_{\lambda _k} = s_{(1^{\lambda _1})}\cdots s_{(1^{\lambda _k})}.$$

We have that $\{m_\lambda \suchthat \lambda \vdash n\}$, $\{s_\lambda \suchthat \lambda \vdash n\}$, $ \{h_\lambda \suchthat \lambda \vdash n\}$ and $\{e_\lambda \suchthat \lambda \vdash n\}$ are all bases of $\Sym ^n$ for $n\geq1$. Additionally, if $f\in \Sym$ is a nonnegative linear combination of elements in these bases, then using the vernacular we say that $f$ is, respectively, \emph{monomial-, Schur-, $h$-} or \emph{$e$-positive}. Also, in the following results, we use the notation {$\prob _n(\cdot\suchthat\cdot )$} to denote that the probability is being calculated in $\Sym ^n$ for $n\geq 1$. Considering its importance, our first result shows the rarity that a monomial-positive symmetric function is furthermore Schur-positive. This statement was previously determined by Bergeron--Patrias--Reiner using a proof method similar to that of Theorem~\ref{thm:genthm}. The statement without proof is given in \cite{Patrias}.

\begin{corollary}\label{cor:smprob}\cite{Patrias} Let $\mathcal{K}_\lambda$ denote the number of SSYTs of shape $\lambda$ and partition content. Then
\[\prob_n(s_\lambda\suchthat m_\lambda)=\prod_{\lambda\vdash n}(\mathcal{K}_\lambda)^{-1}.\]
\end{corollary}
\begin{proof}
The result follows  from Theorem~\ref{thm:genthm} by first setting $A_0=s_{(1^n)}=m_{(1^n)}=B_0$, and ordering the basis elements in increasing order by taking their indices in lexicographic order. Then use Equation~\eqref{eq:sasm} along with $K_{\lambda\mu}=0$ if $\lambda \llex \mu$ and $K_{\lambda\lambda}=1$ \cite[Proposition 7.10.5]{EC2}.\end{proof}

\begin{example}\label{ex:smprob} For $n=3$ we have that $\mathcal{K}_{(3)}=3$, $\mathcal{K}_{(2,1)}=3$ and $\mathcal{K}_{(1,1,1)}=1$ from the following SSYTs. 
$$\tableau{1&1&1}\quad \tableau{1&1&2}\quad \tableau{1&2&3} \qquad \tableau{1&1\\2} \quad \tableau{1&2\\3}\quad \tableau{1&3\\2}\qquad \tableau{1\\2\\3}$$
Hence,
$$\prob _3 (s_\lambda \suchthat m_\lambda)= \left( \frac{1}{3} \right)\left( \frac{1}{3} \right)\left( \frac{1}{1} \right) = \frac{1}{9}.$$
\end{example}

\begin{corollary}\label{cor:sm0}
We have that 
\[\lim_{n\to\infty}\prob_n(s_\lambda\suchthat m_\lambda)=0.\]
\end{corollary}
\begin{proof}
Let $\lambda=(\lambda_1,\ldots,\lambda_k)$ be a partition of $n$, and consider the following two fillings of $\lambda$. For the first, fill the boxes in the top row with $1,\ldots,\lambda_1$ from left to right, the second row with $\lambda_1+1,\ldots,\lambda_1+\lambda_2$, etc. For the second, fill the boxes in row $i$ from the top with $i$ for $1\leq i \leq k$. For $\lambda\neq (1^n)$, these fillings are distinct, and thus $\mathcal{K}_\lambda\geq 2$. It  follows that 

\[0\leq \prod_{\lambda\vdash n}(\mathcal{K}_\lambda)^{-1}\leq \frac{1}{2^{p(n)-1}},\]
where $p(n)$ denotes the number of partitions of $n$, and hence 
\[0\leq \lim_{n\to\infty}\prod_{\lambda\vdash n}(\mathcal{K}_\lambda)^{-1}\leq \lim_{n\to\infty}\frac{1}{2^{p(n)-1}}=0.\]
\end{proof}

\begin{corollary}\label{cor:ehsprob} Let $\mathcal{E}_\lambda$ be the number of SSYTs with content $\lambda$. Then
\[\prob_n(e_\lambda \suchthat s_\lambda)=\prob_n(h_\lambda\suchthat s_\lambda)=\prod_{\lambda\vdash n}(\mathcal{E}_\lambda)^{-1}.\] 
\end{corollary}
\begin{proof}
{By \cite[Proposition 7.10.5 and Corollary 7.12.4]{EC2}} we have that
\begin{equation}\label{eq:hass}h_\lambda = s_\lambda + \sum _{\mu\glex \lambda} K_{\mu\lambda} s_\mu.\end{equation}The result for $\prob _n(h_\lambda\suchthat s_\lambda)$ now follows from Equation~\eqref{eq:hass} and Theorem~\ref{thm:genthm}, {along with $K_{\mu\lambda}=0$ if $\mu\llex \lambda$ and $K_{\lambda\lambda = 1}$ \cite[Proposition 7.10.5]{EC2},} by setting $A_0=h_n=s_{(n)}=B_0$ and by ordering the basis elements in increasing order by taking their indices in reverse lexicographic order. The result for $\prob_n(e_\lambda \suchthat s_\lambda)$ now follows from applying the involution $\omega$ to {that acts as a bijection from Schur-positive functions that are $h$-positive to Schur-positive
functions that are $e$-positive. It satisfies}
$$\omega(h_\lambda)=e_\lambda \mbox{ and } \omega(s_\lambda)=s_{\lambda '}$$where $\lambda ' $ is the transpose of $\lambda$, that is, the partition whose parts are obtained from $\lambda$ with maximum part $\mmax(\lambda)$ by letting $\lambda _i ' =$ the number of parts of $\lambda \geq i$, for $1\leq i \leq \mmax(\lambda)$. 
\end{proof}

\begin{example}\label{ex:ehsprob}
For $n=3$ we have that $\mathcal{E}_{(3)}=1$, $\mathcal{E}_{(2,1)}=2$ and $\mathcal{E}_{(1,1,1)}=4$ from the following SSYTs. 
$$\tableau{1&1&1}\qquad \tableau{1&1&2}\qquad  \tableau{1&1\\2} \qquad \tableau{1&2&3} \qquad \tableau{1&2\\3}\qquad \tableau{1&3\\2}\qquad \tableau{1\\2\\3}$$
Hence,
$$\prob _3 (e_\lambda \suchthat s_\lambda)= \prob _3 (h_\lambda \suchthat s_\lambda)=\left( \frac{1}{1} \right)\left( \frac{1}{2} \right)\left( \frac{1}{4} \right) = \frac{1}{8}.$$
\end{example}

\begin{corollary}\label{cor:ehs0}
We have that \[\lim_{n\to\infty}\prob_n(e_\lambda\suchthat s_\lambda)=\lim_{n\to\infty}\prob_n(h_\lambda\suchthat s_\lambda)=0.\]
\end{corollary}
\begin{proof}
As in the proof of Corollary~\ref{cor:sm0}, the result will follow from showing that $\mathcal{E}_\lambda\geq 2$ for all $\lambda\neq (n)$. Indeed, first consider the tableau $T$ of shape $\lambda=(\lambda_1,\ldots,\lambda_k)$ with the boxes in row $i$ from the top filled with $i$ for $1\leq i \leq k$. Second, consider the tableau of shape $(\lambda_1+1,\lambda_2,\ldots,\lambda_{k-1},\lambda_k-1)$ obtained from $T$ by moving the rightmost box filled with $k$ from row $k$ to row 1. These are distinct for all $\lambda\neq (n)$, hence $\mathcal{E}_\lambda\geq 2$ for all $\lambda\neq (n)$.
\end{proof}

\begin{remark}
One can form a square matrix with the $K_{\lambda\mu}$ (known as the \textit{Kostka numbers}), where $\lambda$ and $\mu$ vary over all partitions of $n$, and rows and columns are ordered in lexicographic order. Then $\mathcal{K}_\lambda$ and $\mathcal{E}_\lambda$ may be interpreted as a row sum and as a column sum of this matrix, respectively.
\end{remark}

Since elementary symmetric functions are Schur-positive, which in turn are monomial-positive, it is natural to compute the following.

\begin{corollary}\label{cor:emprob} Let $\mathcal{M}_\lambda$ be the number of (0,1)-matrices with row sum $\lambda$ and column sum a partition. Then
\[\prob_n(e_\lambda\suchthat m_\lambda)=\prod_{\lambda\vdash n}(\mathcal{M}_\lambda)^{-1}.\]
\end{corollary}
\begin{proof}
Let $A_0=e_n=m_{(1^n)}=B_0$. {Order the basis elements of $A$ in increasing order by taking their indices in reverse lexicographic order. Order the basis elements of $B$ in increasing order by taking the transpose,  as in the proof of Corollary ~\ref{cor:ehsprob}, of their indices in reverse lexicographic order. By \cite[Proposition 7.4.1 and Theorem 7.4.4]{EC2} we have} that
\[e_\lambda={m_{\lambda '}} + \sum_{\mu '\lrlex \lambda}M_{\lambda\mu}m_\mu,\]
where $\mu '$ is the transpose of $\mu$ as in the proof of Corollary ~\ref{cor:ehsprob},  and $M_{\lambda\mu}$ is the number of (0,1)-matrices whose row sums give the parts of $\lambda$ and whose column sums give the parts of $\mu$.
The result now follows from Theorem~\ref{thm:genthm} {along with $M_{\lambda\mu}=0$ if $\lambda \llex \mu '$ and $M_{\lambda\lambda '}=1$ \cite[Theorem 7.4.4]{EC2}}.
\end{proof}

\begin{example}\label{ex:emprob}
For $n=3$ we have that $\mathcal{M}_{(3)}=1$, $\mathcal{M}_{(2,1)}=4$ and $\mathcal{M}_{(1,1,1)}=10$ from the six $3\times 3$ permutation matrices, the matrix $\begin{pmatrix}1&1\\1&0\end{pmatrix}$ and the following four matrices and their transposes. 
$$\begin{pmatrix}1&1&1\end{pmatrix}\quad 
\begin{pmatrix}1&1&0\\0&0&1\end{pmatrix}\quad 
\begin{pmatrix}1&0&1\\0&1&0\end{pmatrix}\quad 
\begin{pmatrix}0&1&1\\1&0&0\end{pmatrix}\quad 
$$
Hence,
$$\prob _3 (e_\lambda \suchthat m_\lambda)= \left( \frac{1}{1} \right)\left( \frac{1}{4} \right)\left( \frac{1}{10} \right) = \frac{1}{40}.$$
\end{example}

\begin{corollary}\label{cor:em0}
We have that 
\[\lim_{n\to\infty}\prob_n(e_\lambda\suchthat m_\lambda)=0.\]
\end{corollary}
\begin{proof}
As before, in the proof of Corollary~\ref{cor:sm0}, the result will follow if we show that $\mathcal{M}_\lambda\geq 2$ for all $\lambda=(\lambda_1,\ldots,\lambda_k)\neq (n)$. Consider the matrix where the first $\lambda_1$ columns have a 1 in row 1 and 0's everywhere else, the next $\lambda_2$ columns have a 1 in row 2 and 0's everywhere else, the next $\lambda_3$ columns have a 1 in row 3 and 0's everywhere else, etc. We obtain a second valid (0,1)-matrix by swapping column $\lambda_1$ with column $\lambda_1+1$.
\end{proof}

\section{Probabilities of quasisymmetric function positivity}\label{sec:probqsym} We now turn our attention to quasisymmetric functions, and again begin by recalling pertinent combinatorial concepts. A \emph{composition} $\alpha = (\alpha _1, \ldots , \alpha _k)$ of $n$, denoted by $\alpha \vDash n$, is a list of positive integers whose \emph{parts} $\alpha _i$ sum to $n$. Observe that every composition $\alpha$ determines a partition $\lambda(\alpha)$, which is obtained by reordering the parts of $\alpha$ {into} weakly decreasing order. Also recall the bijection between compositions of $n$ and subsets of $[n-1] = \{1, \ldots , n-1\}$. Namely, given $\alpha = (\alpha _1, \ldots , \alpha _k) \vDash n$, its corresponding set is $\set (\alpha) = \{\alpha _1, \alpha _1 +\alpha _2, \ldots , \alpha _1 + \cdots + \alpha _{k-1}\} \subseteq [n-1]$. Conversely, given $S=\{s_1, \ldots , s_{k-1}\} \subseteq [n-1]$ its corresponding composition is $\comp(S) = (s_1, s_2 - s_1, \ldots , n-s_{k-1}) \vDash n$. Lastly, the empty set is in bijection with $(n)$. We again use the abbreviation $i^j$ to mean $j$ consecutive parts equal to $i$, and extend the definition of lexicographic order from the previous section for partitions to encompass compositions. We then use this extension to define a total order on compositions of $n$. Given compositions $\alpha, \beta$ we say $\beta \btr \alpha$ if $\lambda (\beta) \llex \lambda (\alpha)$ or $\lambda (\beta) = \lambda (\alpha)$ and $\beta \llex \alpha$.

\begin{example}\label{ex:btr} The compositions of 4 in $\btr$ order are
$$(1^4)\btr (1^2, 2) \btr (1,2,1) \btr (2, 1^2) \btr (2^2) \btr (1,3) \btr (3,1) \btr (4).$$
\end{example}

There is also a partial order on compositions of $n$, which will be useful later. Given compositions $\alpha, \beta$ we say that $\alpha$ is a \emph{proper coarsening} of $\beta$ (or $\beta$ is a \emph{proper refinement} of $\alpha$) denoted by $\beta \refines \alpha$ if we can obtain $\alpha$ by nontrivially adding together adjacent parts of $\beta$. For example, $(1,2,1)\refines (1,3)$. Observe that $\beta \refines \alpha$ if and only if $\set(\alpha) \subset \set (\beta)$.

Now, similar to the previous section, given a composition $\alpha = (\alpha _1 ,\ldots , \alpha _k)$ and commuting variables $\{ x_1, x_2, \ldots \}$ we define the \emph{monomial quasisymmetric function} $M_\alpha$ to be
$$M_\alpha = \sum _{i_1< \cdots <i_k} x_{i_1} ^{\alpha _1}\cdots x_{i_k} ^{\alpha _k}$$and the \emph{fundamental quasisymmetric function} $F_\alpha$ to be
$$F_\alpha = M_\alpha + \sum _{\beta \refines \alpha} M_\beta.$$

\begin{example}\label{ex:mandf} We compute $M_{(2,1)} = x_1^2x_2 + x_1^2x_3+\cdots$ and $F_{(2,1)}=M_{(2,1)}+M_{(1,1,1)}.$
\end{example}

The set of monomial quasisymmetric functions or the set of fundamental quasisymmetric functions forms a basis for the graded algebra of quasisymmetric functions $$\QSym = \bigoplus _{n\geq 0} \QSym ^n \subseteq \bR [[x_1, x_2, \ldots ]]$$where $\QSym ^0 = \spam \{1\}$ and $\QSym ^n  = \spam \{M_\alpha \suchthat \alpha \vDash n\}= \spam \{F_\alpha \suchthat \alpha \vDash n\}$ for $n\geq 1$. Hence each graded piece $\QSym ^n$ for $n\geq 1$ is a finite-dimensional real vector space with basis $\{M_\alpha \suchthat \alpha \vDash n\}$ or $\{F_\alpha \suchthat \alpha \vDash n\}$. 

In order to define our third and final basis of $\QSym$ we need composition diagrams and composition tableaux. Given a composition $\alpha = (\alpha _1, \ldots , \alpha _k) \vDash n$, we call the array of $n$ left-justified boxes with $\alpha _i$ boxes in row $i$ from the top, for $1\leq i \leq k$, the \emph{composition diagram} of $\alpha$, also denoted by $\alpha$. Given a composition diagram $\alpha \vDash n$, we say $\tau$ is a \emph{semistandard composition tableau (SSCT)} of \emph{shape} $\alpha$ if the boxes of $\alpha$ are filled with positive integers such that
\begin{enumerate}
\item the entries in each row weakly decrease when read from left to right,
\item the entries in the leftmost column strictly increase when read from top to bottom,
\item if we denote  the box in $\tau$ that is in the $i$-th row from the top and $j$-th column from the left by $\tau(i,j)$, then if $i<j$ and $\tau(j,m)\leq \tau (i,m-1)$ then $\tau(i, m)$ exists and $\tau(j,m)< \tau(i,m)$.
\end{enumerate}
Furthermore,  if each of the numbers $1, \ldots , n$ appears exactly once, then we say that $\tau$ is a \emph{standard composition tableau {(SCT)}}. Intuitively we can think of the third condition as saying that if $a\leq b$ then $a<c$ in the following array of boxes.
$$\tableau{b&c\\ & \\ & a}$$Given an SSCT $\tau$ we define the \emph{content} of $\tau$, denoted by $\cont(\tau)$, to be the list of nonnegative integers
$$\cont(\tau) = (c_1, \ldots , c_{\mmax})$$where $c_i$ is the number of times that $i$ appears in $\tau$ and $\mmax$ is the largest integer appearing in $\tau$. We say that an SSCT is of \emph{composition content} if $c_i\neq 0$ for all $1\leq i \leq \mmax$. Given an SCT $\tau$ of shape $\alpha \vDash n$, we define its \emph{descent set} to be
$$\des(\tau) = \{ i \suchthat i+1 \mbox{ is weakly right of } i\} \subseteq [n-1]$$and define its \emph{descent composition} to be
$$\comp(\tau) = \comp(\des(\tau)) \vDash n.$$ We can now define our final basis both in terms of monomial and fundamental quasisymmetric functions, respectively.  The first formula is \cite[Theorem 6.1]{QS} with \cite[Proposition 6.7]{QS} applied to it, and the second is \cite[Theorem 6.2]{QS} with \cite[Proposition 6.8]{QS} applied to it.

If $\alpha$ and $\beta$ are compositions, then we define the \emph{quasisymmetric Schur function} $\qs _\alpha$ to be
\begin{equation}\label{eq:qsasm}\qs _\alpha = M_\alpha + \sum _{\beta\btr \alpha} K_{\alpha\beta}^c M_\beta\end{equation}where $K_{\alpha\beta}^c$ is the number of SSCTs, $\tau$, of shape $\alpha$ and  $\cont(\tau)=\beta$. It is also given by
\begin{equation}\label{eq:qsasf}\qs _\alpha = F_\alpha + \sum _{\beta\btr \alpha} d_{\alpha\beta} F_\beta\end{equation}where $d_{\alpha\beta}$  is the number of SCTs, $\tau$, of shape $\alpha$ and  $\comp(\tau)=\beta$.

\begin{example}\label{ex:qsasmf}
$S_{(1,2)} = M_{(1,2)} + M_{(1,1,1)} = F_{(1,2)}$ from the following {SSCT, which is also an SCT, arising from the nonleading term in the first equality}.
$${\tableau{1\\3&2}}$$
\end{example}

We have that, in addition to $\{M_\alpha \suchthat \alpha \vDash n\}$ and  $\{F_\alpha \suchthat \alpha \vDash n\}$,   $\{\qs_\alpha \suchthat \alpha \vDash n\}$ is a basis of $\QSym ^n$ for $n\geq 1$, and if $f\in \QSym$ is a nonnegative linear combination of such basis elements, then we refer to $f$ respectively as being \emph{monomial quasisymmetric-, fundamental-} or \emph{quasisymmetric Schur-positive}. We also use the notation {$\prob _n(\cdot\suchthat\cdot )$} to denote that the probability is being calculated in $\QSym ^n$ for $n\geq 1$. Our first result is reminiscent of the probability that a monomial-positive symmetric function is furthermore Schur-positive in Corollary~\ref{cor:smprob}.

\begin{corollary}\label{cor:SMprob} Let $\mathcal{K}_\alpha^c$  be the number of SSCTs of shape $\alpha$ and composition content. Then
\[\prob_n(\qs_\alpha \suchthat M_\alpha)=\prod_{\alpha\vDash n}(\mathcal{K}_\alpha^c)^{-1}.\]\end{corollary}

\begin{proof} The result follows  from Theorem~\ref{thm:genthm} by {first} setting $A_0=\qs_{(1^n)}=M_{(1^n)}=B_0$, {and} ordering the basis elements in increasing order by taking their indices in $\btr$ order. {Then use Equation~\eqref{eq:qsasm} along with $K_{\alpha\beta}^c=0$ if $\alpha \btr \beta$ and $K_{\alpha\alpha}^c = 1$ \cite[Proposition 6.7]{QS}}.
\end{proof}

\begin{example}\label{ex:SMprob}
For $n=3$ we have that $\mathcal{K}^c_{(3)}=4$, $\mathcal{K}^c_{(2,1)}=2$, $\mathcal{K}^c_{(1,2)}=2$ and $\mathcal{K}^c_{(1,1,1)}=1$ from the following SSCTs. 
$$\tableau{1&1&1}\quad \tableau{2&1&1}\quad \tableau{2&2&1}\quad \tableau{3&2&1} \qquad \tableau{1&1\\2} \quad \tableau{2&1\\3}\qquad \tableau{1\\2&2} \quad \tableau{1\\3&2}\qquad \tableau{1\\2\\3}$$
Hence,
$$\prob _3 (\qs_\lambda \suchthat M_\lambda)= \left( \frac{1}{4} \right)\left( \frac{1}{2} \right)\left( \frac{1}{2} \right)\left( \frac{1}{1} \right) = \frac{1}{16}.$$
\end{example}

\begin{corollary}\label{cor:SM0} We have that
\[\lim_{n\to\infty}\prob_n(\mathcal{S}_\alpha\suchthat M_\alpha)=0.\]
\end{corollary}
\begin{proof}
Let $\alpha = (\alpha _1, \ldots , \alpha _k)$ be a composition of $n$ and consider the following two fillings of $\alpha$. For the first, fill the boxes of $\alpha$ such that the boxes in the bottom row contain $n, n-1, \ldots , n+1-\alpha _k$ from left to right, the next row up $n-\alpha _k, n-\alpha _k -1,  \ldots, n+1-\alpha_k-\alpha_{k-1}$ etc. For the second, fill the boxes in row $i$ from the top with $i$, for $1\leq i \leq k$. For $\alpha \neq (1^n)$ these fillings are distinct, and thus $\mathcal{K}_\alpha^c\geq 2$. Since the number of compositions of $n$ is $2^{n-1}$ it follows that
$$0\leq \lim_{n\to\infty}\prod_{\alpha\vDash n}(\mathcal{K}_\alpha^c)^{-1} \leq \lim_{n\to\infty}\frac{1}{2^{2^{n-1}-1}}=0.$$
\end{proof}

\begin{corollary} \label{cor:SFprob} Let $\mathcal{D}_\alpha$ be the number of SCTs of shape $\alpha$. Then
\[\prob_n(\qs_\alpha\suchthat F_\alpha)=\prod_{\alpha\vDash n}(\mathcal{D}_\alpha)^{-1}.\] \end{corollary}
\begin{proof}
The result follows  from Theorem~\ref{thm:genthm} by {first} setting $A_0=\qs_{(1^n)}=F_{(1^n)}=B_0$ {and} ordering the basis elements in increasing order by taking their indices in $\btr$ order. {Then use Equation~\eqref{eq:qsasf} along with $d_{\alpha\beta}=0$ if $\alpha \btr \beta$ and $d_{\alpha\alpha} = 1$ \cite[Proposition 6.8]{QS}}.
\end{proof}

\begin{example}\label{ex:SFprob}
For $n=3$ we have that $\mathcal{D}_{(3)}=1$, $\mathcal{D}_{(2,1)}=1$, $\mathcal{D}_{(1,2)}=1$ and $\mathcal{D}_{(1,1,1)}=1$ from the following SSCTs. 
$$\tableau{3&2&1} \qquad  \tableau{2&1\\3} \qquad  \tableau{1\\3&2}\qquad \tableau{1\\2\\3}$$
Hence,
$$\prob _3 (\qs_\lambda \suchthat F_\lambda)= 1.$$
\end{example}

\begin{corollary}\label{cor:SF0}
We have that 
\[\lim_{n\to\infty}\prob_n(\mathcal{S}_\alpha\suchthat F_\alpha)=0.\]
\end{corollary}
\begin{proof}
By \cite[Theorem 4.4]{BvWmultiplicity}, we know that $\mathcal{S}_\alpha=F_\alpha$ if and only if $\alpha=(m,1^{\epsilon_1},2,1^{\epsilon_2},\ldots ,2,1^{\epsilon_{k}})$, where $m\in\mathbb{N}_0 {=\{0,1,2,\ldots \}}$ ($m=0$ is understood to mean it does not appear in the composition), $k\in\mathbb{N}_0$, $\epsilon_i\in\mathbb{N}{=\{1,2,\ldots \}}$ for $i\in[k-1]$, and $\epsilon_k\in\mathbb{N}_0$.  

{Let $\mathcal{A}_n$ be the set of compositions of $n$ \emph{not} in the set of compositions described above. Note that if $\alpha\in \mathcal{A}_n$ then $\mathcal{D}_\alpha \geq 2$. Also note that if $\alpha = (\alpha _1, \ldots , \alpha _k) \in \mathcal{A}_n$, then $(\alpha _1, \ldots , \alpha _k +1), (\alpha _1, \ldots , \alpha _k,1) \in \mathcal{A}_{n+1}$. Hence $2|\mathcal{A}_n | \leq |\mathcal{A}_{n+1}|$. Using this repeatedly, along with $| \mathcal{A}_4 |=2$ since $\mathcal{A}_4 = \{(1,3),(2,2)\}$, yields that for $n\geq 5$}
{$$2^{n-3}\leq | \mathcal{A}_n | .$$}{Hence it follows that}
{\[0\leq \lim_{n\to\infty}\prod_{\alpha\vDash n}(\mathcal{D}_\alpha)^{-1}\leq \lim_{n\to\infty}\frac{1}{2^{|\mathcal{A}_n|}}\leq \lim_{n\to\infty}\frac{1}{2^{2^{n-3}}}=0.\]}
\end{proof}

We end with the most succinct of our formulas, namely the probability that a quasisymmetric monomial-positive function is furthermore fundamental-positive.

\begin{corollary}\label{cor:FMprob}  
\[\prob_n(F_\alpha\suchthat M_\alpha)=\frac{1}{(n-1)2^{n-2}}\]
\end{corollary}
\begin{proof}
Recall that \[F_\alpha = M_\alpha + \sum _{\beta \refines \alpha} M_\beta,\]
and that $\beta\prec\alpha$ if and only if $\set(\alpha)\subset\set(\beta)$. Letting $A_0=F_{(1^n)}=M_{(1^n)}=B_0$ and ordering the basis elements in increasing order by taking their indices in $\btr$ order, Theorem~\ref{thm:genthm} gives that 
\[\prob_n(F_\alpha\suchthat M_\alpha)=\prod_{\alpha\vDash n}\left(\sum_{\beta\preceq\alpha}1\right)^{-1}.\] Now 
\[\prod_{\alpha\vDash n}\left(\sum_{\beta\preceq\alpha}1\right) = 
\prod_{S\subseteq [n-1]}\left(\sum_{T\supseteq S}1\right)= \prod_{S\subseteq [n-1]}\left(2^{n-1-|S|}\right)=2^{\sum_{S\subseteq[n-1]}(n-1-|S|)},\]and 
 \[\sum_{S\subseteq[n-1]}(n-1-|S|)=\sum_{T\subseteq[n-1]}|T|=\sum_{k=0}^{n-1}k\binom{n-1}{k}=(n-1)2^{n-2},\] where the last equality is \cite[{Chapter 1 Exercise 2(b)}]{EC1}.
\end{proof}

\begin{example}\label{ex:FM} For $n=3$ we have that $\prob _3 (F_\alpha\suchthat M_\alpha) = \frac{1}{4}.$
\end{example}

The following corollary follows  from Corollary~\ref{cor:FMprob}.
\begin{corollary}\label{cor:FM0} We have that
\[\lim_{n\to\infty}\prob_n(F_\alpha\suchthat M_\alpha)=0.\]
\end{corollary}

\section*{Acknowledgements}\label{sec:acknow} The authors would like to thank Fran\c{c}ois Bergeron, Vic Reiner, and S\'{e}bastien Labb\'{e} for conversations that led to fruitful research directions, and LaCIM where some of the research took place. The first author received support from  the National Sciences and Engineering Research Council of Canada, CRM-ISM, and the Canada Research Chairs Program. The second author was supported in part by the National Sciences and Engineering Research Council of Canada, the Simons
Foundation, and the Centre de Recherches Math\'{e}matiques, through the Simons-CRM scholar-in-residence program.


\def\cprime{$'$}


\begin{thebibliography}{10}

\bibitem{AlexSulz}
{{\sc P.~Alexandersson and R.~Sulzgruber},
{\em $P$-partitions and $p$-positivity},
arXiv:1807.02460v2.}



\bibitem{BBR06}
{\sc F.~Bergeron, R.~Biagioli and  M.~Rosas}, 
{\em Inequalities between {L}ittlewood-{R}ichardson coefficients},
J. Combin. Theory Ser. A 113 (2006) 567--590.

\bibitem{BLvW} 
{\sc C.~Bessenrodt, K.~Luoto and S.~van Willigenburg}, 
{\em Skew quasisymmetric {S}chur functions and noncommutative {S}chur functions},
Adv. Math.
226 (2011)
 4492--4532.
 
\bibitem{SSQSS} 
{\sc C.~Bessenrodt, V.~Tewari and S.~van Willigenburg}, 
{\em {L}ittlewood-{R}ichardson rules for symmetric skew quasisymmetric {S}chur functions}, 
J. Combin. Theory Ser. A 137 (2016) 179--206.

\bibitem{BvWmultiplicity}
{\sc C.~Bessenrodt  and S.~van Willigenburg}, 
{\em Multiplicity free {S}chur, skew {S}chur, and quasisymmetric {S}chur functions},   
Ann. Comb.  17 (2013) 275--294. 

\bibitem{CarlssonMellit}
{\sc E.~Carlsson and A. Mellit}, 
\emph{A proof of the shuffle conjecture}, 
J. Amer. Math. Soc. 31 (2018) 661--697.

\bibitem{DKLT} 
{\sc G.~Duchamp, D.~Krob, B.~Leclerc and J-Y.~Thibon}, 
{\em Fonctions quasi-sym\'etriques, 
fonctions sym\'etriques non-commutatives, et alg\`ebres de Hecke \`a $q = 0$},
C. R. Math. Acad. Sci. Paris
 322 (1996) 
  107--112.
  
\bibitem{FFLP05}
{\sc S.~Fomin, W.~Fulton, C.-K.~Li and Y.-T.~Poon}, {\em Eigenvalues, singular values, and {L}ittlewood-{R}ichardson coefficients},
Amer. J. Math. 127 (2005) 101--127.




\bibitem{Gasharov} 
{\sc V.~Gasharov}, {\em Incomparability graphs of (3+1)-free posets are $s$-positive}, Discrete Math.
157 (1996) 193--197.

\bibitem{HHLRU}
{\sc J.~Haglund, M.~Haiman, N.~Loehr, J.~Remmel and A.~Ulyanov},
\emph{A combinatorial formula for the character of the diagonal coinvariants}, 
Duke Math. J. 126 (2005) 195--232.

\bibitem{QS}
{\sc J.~Haglund, K.~Luoto, S.~Mason and S.~van Willigenburg}, 
{\em {Quasisymmetric {S}chur functions}}, 
J. Combin. Theory Ser. A 118 (2011) 463--490. 

\bibitem{KWvW08}
{\sc R.~King, T.~Welsh and S.~van Willigenburg},  
{\em Schur positivity of skew {S}chur function differences and applications to ribbons and {S}chubert classes},
J. Algebraic Combin. 28 (2008) 139--167.


\bibitem{Kir04}
{\sc A.~Kirillov},
    {\em An invitation to the generalized saturation conjecture},
Publ. Res. Inst. Math. Sci.
  40 (2004) 1147--1239.


\bibitem{lpp}
{\sc T.~Lam, A.~Postnikov and P.~Pylyavskyy}, 
{\em {{S}chur positivity and {S}chur log-concavity}}, 
Amer. J. Math. 129 (2007) 1611--1622.

\bibitem{LamP}
{\sc T.~Lam and P.~Pylyavskyy}
{\em $P$-partition products and fundamental quasi-symmetric
function positivity},
Adv. in Appl. Math. 40 (2008) 271--294. 

              

\bibitem{McN08}
{\sc P.~McNamara},  {\em Necessary conditions for {S}chur-positivity}, J. Algebraic Combin. 28 (2008) 495--507.

\bibitem{McN14}
{{\sc P.~McNamara},  
{\em Comparing skew {S}chur functions: a quasisymmetric perspective}, 
J.  Comb. 5 (2014) 51--85.}

\bibitem{McvW09b}
{\sc P.~McNamara and S.~van Willigenburg}, 
{\em Positivity results on ribbon {S}chur function differences}, European J. Combin. 30 (2009) 1352--1369.



\bibitem{Oko97}
{\sc A.~Okounkov},  {\em Log-concavity of multiplicities with application to characters of {${\rm U}(\infty)$}}, 
Adv. Math. 127 (1997) 258--282.

\bibitem{Patrias}
{\sc R.~Patrias},
{\em What is {S}chur positivity and how common is it?},
arXiv:1809.04448v1.

              
\bibitem{EC1}
{\sc R.~Stanley},
\emph{Enumerative Combinatorics. Volume 1},
Wadsworth  and Brooks/Cole (1986).

\bibitem{EC2}
{{\sc R.~Stanley},
\emph{Enumerative Combinatorics. Volume 2},
Cambridge University Press (1999).}
              
\bibitem{StanleyStembridge} 
{\sc R.~Stanley and J.~Stembridge}, 
{\em On immanants of {J}acobi-{T}rudi matrices and permutations with restricted position}, 
{J. Combin. Theory Ser. A} 62 (1993) 261--279.

\bibitem{0Hecke}
{\sc V.~Tewari and S.~van Willigenburg}, 
{\em Modules of the 0-{H}ecke algebra and quasisymmetric {S}chur functions}, 
Adv. Math. 285 (2015) 1025--1065.

\end{thebibliography}
\end{document}